\documentclass[11pt,a4paper]{article}
\usepackage[T1]{fontenc}
\usepackage{lmodern,microtype}
\usepackage[margin=27mm]{geometry}
\usepackage{amsmath,amssymb,amsthm,mathtools}
\usepackage[numbers,square,sort&compress]{natbib}
\usepackage[hidelinks]{hyperref}
\usepackage{xcolor}
\hypersetup{pdftitle={An Upper Bound for the Mean Speed of Transition Fronts and Unbounded Front Widths for Fisher-KPP Equations in Almost Periodic Media}}

\newcommand{\R}{\mathbb R}
\newcommand{\Z}{\mathbb Z}
\newcommand{\N}{\mathbb N}
\newcommand{\Q}{\mathbb Q}
\newcommand{\T}{\mathbb T}
\newcommand{\Sspace}{\mathcal S}
\newcommand{\D}{\mathcal D}
\newcommand{\Lop}{\mathcal L}
\newcommand{\hull}{H(a)}
\newcommand{\amin}{\underline a}
\newcommand{\amax}{\overline a}
\newcommand{\wmax}{\widehat w}

\newcommand{\ls}{x^-_{1-s}}
\newcommand{\rs}{x^+_s}
\newcommand{\mass}[1]{\int_{\Sspace}#1\,dm}
\newcommand{\ip}[2]{\langle #1,#2\rangle}
\newcommand{\norm}[1]{\lVert #1\rVert}
\DeclareMathOperator{\supp}{supp}
\numberwithin{equation}{section}
\newtheorem{theorem}{Theorem}[section]
\newtheorem{lemma}[theorem]{Lemma}
\newtheorem{proposition}[theorem]{Proposition}
\newtheorem{corollary}[theorem]{Corollary}
\theoremstyle{definition}
\newtheorem{definition}[theorem]{Definition}
\theoremstyle{remark}
\newtheorem{remark}[theorem]{Remark}
\allowdisplaybreaks[1]
\title{An Upper Bound for the Mean Speed of Transition Fronts and Unbounded Front Widths for Fisher--KPP Equations in Almost Periodic Media}
\author{Xing Liang \thanks{School of Mathematical Sciences, University of Science and Technology of China, Hefei, Anhui 230026, China,
\textit{Email address}: \texttt{xliang@ustc.edu.cn}; supported by NSFC 12331006 and 12531008, corresponding author} 
\quad Linfeng Xu \thanks{School of Mathematical Sciences, University of Science and Technology of China, Hefei, Anhui 230026, China,
\textit{Email address}: \texttt{xlf\_hzfy@mail.ustc.edu.cn}}
\quad Qi Zhou \thanks{Chern Institute of Mathematics and LPMC, Nankai University, Tianjin
300071, China,
\textit{Email address}: \texttt{qizhou@nankai.edu.cn}, supported by NSFC 12531006 and 12526201 and Nankai Zhide Foundation} 
\quad Tao Zhou \thanks{Center for Pure Mathematics, School of Mathematical Science, Anhui University, Hefei, Anhui 230601, China,
\textit{Email address}: \texttt{tzhou910@ustc.edu.cn}, supported by NSFC 12471150} }
\date{}
\begin{document}
\maketitle

\begin{abstract}
In this paper, we investigate transition fronts and spreading solutions of Fisher--KPP equations in one-dimensional almost periodic media, both on the real line and on the lattice. Let $\lambda_1$ be the supremum of the spectrum of the linearized operator acting on $L^2(\R)$ or $\ell^2(\Z)$, respectively, and let 
\(L(\lambda_1)\) denote the spatial Lyapunov exponent at the spectral parameter \(\lambda_1\). We prove that, if $L(\lambda_1)>0$, the global mean speed of any transition front is at most $\lambda_1/L(\lambda_1)$. Moreover, if a solution of the Cauchy problem spreads faster than this bound, its transition width is unbounded along a sequence of times. This occurs for a class of initial data with slowly decaying exponential tails. This contrasts with periodic media, where pulsating fronts
exist at every speed above the minimal speed.
As applications, we consider almost Mathieu coefficients and a continuous quasiperiodic coefficient with two frequencies. Together with the results of Nadin--Rossi \citep{NR17}
and Liang--Wang--Zhou--Zhou \citep{LWZZ24},
our results provide an almost complete picture of the
admissible speeds of generalized transition fronts,
while leaving the critical cases open. \end{abstract}

\section{Introduction}

\subsection{Background and the main questions}

We consider the Fisher--KPP equation
\begin{equation}\label{eq:intromodel}
 u_t=\D u+a(x)u(1-u),\qquad x\in\Sspace,
\end{equation}
where $\Sspace=\R$ or $\Z$, $a$ is an almost periodic function with $\inf a>0$, and
\[
 \D u(x)=
 \begin{cases}
 u_{xx}(x),&\Sspace=\R,\\
 u(x+1)+u(x-1)-2u(x),&\Sspace=\Z.
 \end{cases}
\]
Our purpose is to study the relation between the speeds of transition fronts and the spreading of solutions with slowly decaying initial data.

Since the works of Fisher and Kolmogorov, Petrovskii and Piskunov \citep{Fisher,KPP}, traveling fronts and spreading speeds have been studied extensively. For $a\equiv a_0>0$, equation \eqref{eq:intromodel} on $\R$ has a traveling front $u(t,x)=U(x-wt)$ connecting $1$ to $0$ if and only if $w\geq2\sqrt{a_0}$. The corresponding minimal speed on $\Z$ is
\[
 \inf_{p>0}\frac{2(\cosh p-1)+a_0}{p}.
\]
{In periodic media, traveling fronts are replaced by pulsating
	fronts of the form $u(t,x)=U(x-wt,x)$, with $U$ periodic in
	its second variable. Early developments include the work of
	Shigesada, Kawasaki and Teramoto \citep{SKT86} on population
	spread and that of Xin \citep{Xin92} on flame propagation.
	For periodic Fisher--KPP equations, pulsating fronts exist
	at and above a minimal speed \citep{BH02,GH06}.
	The corresponding spreading theory relates this minimal
	speed to invasion from localized initial data in one space
	dimension; see \citep{Weinberger,BHN05}.
	Liang and Zhao \citep{LiangZhao07} established an abstract
	framework for spreading speeds and traveling waves of
	monotone semiflows, subsequently extended to evolution
	systems with spatial structure, including periodic habitats
	\citep{LiangZhao10}.}
	
{Further work has clarified the qualitative properties of
	fronts and their speeds. Berestycki, Hamel and Nadirashvili
	\citep{BHN05} derived spectral characterizations and studied
	the effects of diffusion, advection and geometry, while
	Liang, Lin and Matano \citep{LLM10} investigated the
	optimization of the minimal speed under a constraint on
	the average reaction rate. Hamel and Roques \citep{HR11}
	proved uniqueness and stability of pulsating fronts,
	including convergence from suitable noncompact front-like
	initial data. In particular, appropriate exponential tails
	can select speeds above the minimal one. Slower tails may
	instead produce accelerating propagation, as shown by
	Hamel and Roques \citep{HR10} in homogeneous media and
	Henderson \citep{Henderson16} in periodic media.}
	
{A natural question is which aspects of this propagation
	theory persist in more general heterogeneous environments.
	Understanding the relation between spreading speeds and
	the existence of fronts with uniformly bounded transition
	width is an important part of this question. Beyond
	periodic geometry, Berestycki, Hamel and Nadirashvili
	\citep{BHN10} studied spreading from compactly supported
	initial data in general unbounded domains, introducing
	different notions of spreading speed and analyzing their
	dependence on the domain and the initial condition.
	Almost periodic media provide a natural setting in which
	to investigate the effect of nonperiodic coefficients:
	they retain spatial recurrence but need not possess an
	exact period.}

For nonperiodic media, several notions of traveling fronts have been introduced. Matano's definition \citep{Matano03} describes recurrence of the profile through its dependence on the translated medium. For example, in the continuous case it has the form
\begin{equation}\label{eq:matanointro}
 u(t,x)=\mathcal P(a(\cdot+X(t)),x-X(t)),
 \qquad \mathcal P:\hull\times\R\longrightarrow[0,1],
\end{equation}
where $\hull$ is the uniform hull of $a$, the map $\mathcal P$ is continuous, and its limits at the two spatial infinities are {steady states $1$ and $0$ respectively}, uniformly in the hull variable. A lattice formulation using a passage-time cocycle is given in \citep{LWZZ24}. The generalized transition fronts of Berestycki and Hamel \citep{BH12} require only uniform convergence to the two {steady states} away from an interface; recurrence and spatial monotonicity are not assumed.It can be verified that almost periodic traveling wave solutions in Matano's sense are transition fronts in the sense of Berestycki–Hamel; Matano's definition particularly reflects the close connection between the propagation properties of the wave and the recurrence of the spatial medium. We mainly use the notion of generalized transition fronts, in the one-dimensional single-interface setting, together with a positive global mean speed. The precise definition is given in Section~\ref{sec:setting}. Time almost periodic equations have a different recurrence structure; see \citep{HS09,Shen10,Shen11}. Here the coefficients are independent of time, and all Lyapunov exponents refer to the spatial linearized equation.

Generalized principal eigenvalues of the linearized operator
\[
 \Lop_a:=\D+a
\]
provide a useful approach to spreading in almost periodic media
\citep{BN12,LZ20}. Nadin and Rossi \citep{NR17} and Liang, Wang,
Zhou and Zhou \citep{LWZZ24} develop existence theories for
transition fronts in the continuous and lattice settings,
respectively. In the present self-adjoint setting, the relevant
speeds are given by
\begin{equation}\label{eq:introinterval}
 w^*=\inf_{E>\lambda_1}\frac{E}{\mu(E)},\qquad
 \overline w=\frac{\lambda_1}{\mu_0},\qquad
 \mu_0=\lim_{E\downarrow\lambda_1}\mu(E),
\end{equation}
where $\lambda_1$ is the supremum of the spectrum of $\Lop_a$
and $\mu(E)$ is the spatial decay rate of the positive solution
of $\Lop_a\phi=E\phi$ that decays at $+\infty$.
They show that 
transition fronts exist with every global mean speed
$w\in[w^*,\overline w)$, whereas no transition front has
global mean speed $w<w^*$; see
\citep[Theorem~1.5]{NR17} and \citep[Theorem~1.1]{LWZZ24}.
A positive almost periodic eigenfunction at $\lambda_1$
implies $\overline w=+\infty$.

When $\overline w$ is finite, these results leave open the
existence of fronts at or above the upper endpoint.
In the present work, we prove that every transition front with positive global
mean speed $w$ satisfies
\begin{equation}\label{eq:introbound}
 w\leq\wmax:=\frac{\lambda_1}{L(\lambda_1)},
 \qquad L(\lambda_1)>0,
\end{equation}
where $L(E)$ is the spatial Lyapunov exponent.
As recalled in Section~\ref{sec:setting},
$\mu_0=L(\lambda_1)$, so $\wmax=\overline w$.
Together with the results of \citep{NR17,LWZZ24},
this completes the existence--nonexistence picture for
transition fronts with a global mean speed in the settings
covered by those results, apart from the upper endpoint
$w=\wmax$.
The finite upper speed bound when $L(\lambda_1)>0$
contrasts sharply with spatially homogeneous and periodic
media, where traveling and pulsating fronts, respectively,
exist at every speed at or above the minimal speed.

Quasiperiodic Schr\"odinger operators at large coupling provide a natural
source of examples. For a real-analytic sampling function
$v:\T^d\to\R$, a frequency vector $\alpha\in\R^d$ and a phase
$\theta\in\T^d$, write
\begin{equation}\label{eq:introSchrodinger}
 (H_{\kappa,\alpha,\theta}^{v}\varphi)(n)
 =\varphi(n+1)+\varphi(n-1)
   +\kappa v(\theta+n\alpha)\varphi(n),\qquad n\in\Z.
\end{equation}
For $a_\theta(n)=A+\kappa v(\theta+n\alpha)$, with
$A>\kappa\norm{v}_\infty$, we write
$\Lop_\theta:=\Lop_{a_\theta}$, so that
$\Lop_\theta=H_{\kappa,\alpha,\theta}^{v}+(A-2)I$.
For our purposes, the relevant input is positivity of the Lyapunov
exponent: for nonconstant analytic $v$, \citep{BG00,SS91} give a positive lower bound at every
energy when $\kappa$ is sufficiently large. The largeness of  $\kappa$ is essential, otherwise the Lyapunov exponents are zero in the spectrum \citep{Avila15,Eliasson92}. 
Continuous examples follow from the low-energy estimates of
Bjerkl\"ov \citep{Bjerklov06} and You-Zhou
\citep{YZ14}.

A positive localized eigenfunction at the spectral edge, when one
exists, suggests an obstruction to the slowly decaying leading edge
of a fast front. Such an eigenfunction is not an assumption of our
argument and is not implied by a positive Lyapunov exponent.
Indeed, in an ergodic discrete Schr\"odinger family, each fixed
energy is an eigenvalue for a set of phases of measure zero
\citep[Theorem~3.6]{Damanik}. In particular, the phase-independent
spectral edge is not an eigenvalue for almost every phase, even
in a regime of Anderson localization. This assertion concerns a
fixed energy; it does not assert absence of phase-dependent
eigenvalues.

There are also examples with no eigenvalues at any energy.
For an even sampling function, Jitomirskaya and Simon \citep{JS94}
prove absence of point spectrum for a dense $G_\delta$ set of phases.
 Moreover, sufficiently strong
rational approximation of the frequency can exclude eigenvalues
for every phase by sharp Gordon's argument \citep{AYZ17,Gordon76}.
Our conclusions apply to these singular continuous examples as well.

\subsection{Spreading solutions and their transition widths}\label{sec:introsharp}

The second question concerns the Cauchy problem with front-like initial data. 
In homogeneous media, slowly decaying exponential tails may select speeds above the minimal wave speed while preserving bounded transition width. More precise information is available from sharp level-set estimates. 
Under their respective hypotheses, Alfaro, Giletti and Xiao \citep{AGX25} treat critical tails $u_0(x)\asymp x^k e^{-x}$, $k\geq-2$, with $f'(0)=1$, and Zhang \citep{Zhang25} also treats $u_0(x)\asymp x^q e^{-px}$, $q\in\mathbb R$ and $0<p<\sqrt{f'(0)}$. 
In these cases every fixed level remains at bounded distance from a common position $m(t)$, given respectively by
\begin{equation}\label{eq:sharpcenters}
 \begin{aligned}
 m_k(t)&=
 \begin{cases}
 2t+\dfrac{k-1}{2}\log t+O_{t\to+\infty}(1),&k>-2,\\[2pt]
 2t-\dfrac32\log t+\log\log t+O_{t\to+\infty}(1),&k=-2,
 \end{cases}\\
 m_{p,q}(t)&=\left(p+\frac{f'(0)}p\right)t+\frac{q}{p}\log t+O_{t\to+\infty}(1).
 \end{aligned}
\end{equation}
Here $\asymp$ denotes two-sided comparison for large $x$. The common displacement cancels when two levels are compared, giving bounded width. Related estimates for steep initial data are obtained in \citep{Bramson,HNRR16}. 
{In the continuous almost periodic setting, Liang, Xu and
	Zhou \citep{LXZ26} establish spreading speeds and first-order
	level-set estimates for classes of exponentially and
	subexponentially decaying initial data. 
	These results yield that every fixed transition level is located
	at $ct+o(t)$, where $c$ is the spreading speed.
	These first-order estimates do not determine whether
	the transition width remains bounded.}

In this paper, we show that, under $L(\lambda_1)>0$, a rightmost level propagating faster than $\wmax$ cannot have an eventually bounded transition width. In particular, in the spatially continuous case the initial datum $u_0(x)=e^{-px_+}$ with $0<p<L(\lambda_1)$ gives such a solution, although its {rightward spreading speed} is finite. We prove unboundedness along a sequence of times, not divergence of the width or a precise growth rate.
{
In this regime, there can be no common centering function
$m(t)$ for which all transition levels remain within
$O(1)$ of $m(t)$, since such estimates would imply bounded
transition width.
Thus an analogue of the Bramson-type description with
a common center and bounded errors for all transition
levels is ruled out, regardless of the lower-order
correction included in $m(t)$.
}

This conclusion also differs from the localized-perturbation problem in \citep{NRRZ12}. A favorable compactly supported perturbation can admit slow-tail Cauchy solutions with arbitrarily large finite speeds and bounded width for $t\geq0$, even when entire fronts at those speeds do not exist. We will explain this point in Remark~\ref{rem:localizedwidth}.

\subsection{Outline of the proof}

The proof has two main ingredients. A uniform local exponential growth estimate gives an exponential upper estimate on the leading edge. At almost every element $g$ of the hull of $a$, two positive half-line solutions of the spectral-edge equation can be joined to obtain a decaying weight $\psi$ satisfying
\[
 (\lambda_1-\Lop_g)\psi=d\delta_0,\qquad d>0.
\] Here $\delta_0$ denotes the unit mass at the origin: the Dirac
measure in the continuous case and the sequence
$\delta_0(n)=\mathbf{1}_{\{0\}}(n)$ in the lattice case.

A hypothetical fast front then yields a positive ancient solution of the linearized equation with finite weighted mass. A time-monotonicity theorem contradicts the mass identity associated with $\psi$. For the Cauchy problem, we select first arrival times of the interface at prescribed spatial positions. This selection preserves a uniform bound on past interface positions and places the limiting coefficient in the full-measure set where the weight is available.

Section~\ref{sec:setting} states our assumptions and main results. Sections~\ref{sec:local}--\ref{sec:ancient} contain the estimates for the nonlinear and linearized equations. We prove the front-speed bound in Section~\ref{sec:cone}, and the width theorem in Section~\ref{sec:cauchy}. Applications and further questions are discussed in Sections~\ref{sec:examples} and \ref{sec:scope}.

\section{Assumptions and main results}\label{sec:setting}

\subsection{The hull and the linearized operator}

Let $\Sspace=\R$ or $\Z$, with Lebesgue or counting measure $m$, respectively. We assume throughout that $a$ is Bohr almost periodic and
\begin{equation}\label{eq:apos}
 0<\amin:=\inf_{x\in\Sspace}a(x)
 \leq\amax:=\sup_{x\in\Sspace}a(x)<+\infty.
\end{equation}
When $\Sspace=\R$, we also assume $a\in C_b^{0,\alpha}(\R)$ for some $\alpha\in(0,1)$. Set
\[
 \hull=\overline{\{a(\cdot+y):y\in\Sspace\}}_{\|\cdot\|_\infty},
 \qquad \tau_y g=g(\cdot+y).
\]
The translation action on $\hull$ is minimal and uniquely ergodic. Its invariant probability measure is denoted by $\mu_H$, to distinguish it from the spatial decay rate $\mu(E)$. For each $g\in\hull$, consider
\begin{equation}\label{eq:kpp}
 u_t=\D u+g(x)u(1-u),\qquad x\in\Sspace,
\end{equation}
and write $\Lop_g=\D+g$ for its linearization at zero. This is a self-adjoint operator on $L^2(\Sspace,m)$, with domain $H^2(\R)$ or $\ell^2(\Z)$. Denote its spectrum by $\sigma(\Lop_g)$. Minimality implies that
\[
 \lambda_1:=\sup\sigma(\Lop_g)
\]
is independent of $g$, and $\lambda_1\geq\amin>0$; see Lemma~\ref{lem:boxes}.

For $\Sspace=\Z$, let $\Phi_E(n,g)$ be the transfer product associated with
\[
 A_E(g)=\begin{pmatrix}E+2-g(0)&-1\\1&0\end{pmatrix},
 \qquad \binom{v(n+1)}{v(n)}=A_E(\tau_n g)\binom{v(n)}{v(n-1)}.
\]
For $\Sspace=\R$, let $\Phi_E(x,g)$ be the fundamental matrix, normalized at zero, of
\[
 \frac{d}{dx}\binom{v}{v'}=
 \begin{pmatrix}0&1\\E-g(x)&0\end{pmatrix}\binom{v}{v'}.
\]
The spatial Lyapunov exponent is
\begin{equation}\label{eq:gamma}
 L(E)=\lim_{r\to+\infty}\frac1r
 \int_{\hull}\log\norm{\Phi_E(r,g)}\,d\mu_H(g),
\end{equation}
where $r\in\N$ in the lattice case.

Recall the definition of $\mu(E)$ in \eqref{eq:introinterval}. 
For $E>\lambda_1$, the corresponding positive solution spans
the stable direction, and hence $\mu(E)=L(E)$; see \citep[Proposition~1.3 and Lemma~2.4]{NR17} in the continuous case and \citep[Proposition~3.4]{LWZZ24} in the lattice case.

Let $\nu$ denote the density-of-states measure on $\R$ associated
with the family $\{\Lop_g\}_{g\in\hull}$. More precisely, let
$\lambda_{R,j}(g)$ be the eigenvalues, counted with multiplicity,
of the Dirichlet restriction of $\Lop_g$ to
$I_R=[-R,R]\cap\Sspace$. Then
\[
 \int_\R F(\xi)\,d\nu(\xi)
 =\lim_{R\to+\infty}\frac{1}{m(I_R)}
   \sum_j F\bigl(\lambda_{R,j}(g)\bigr),
 \qquad F\in C_c(\R).
\]
In the present almost periodic setting, this limit is independent
of $g$. Thus $\nu$ describes the limiting distribution of
eigenvalues per unit length, or per lattice site. It is a locally
finite measure, with total mass one in the lattice case; no
absolute continuity is assumed.

Fix $E_0>\lambda_1$. The Thouless formula, in its relative
form in the continuous case, gives
\[
 L(E_0)-L(E)
 =\int_{(-\infty,\lambda_1]}
   \log\frac{E_0-\xi}{E-\xi}\,d\nu(\xi),
 \qquad \lambda_1<E<E_0;
\]
see \citep[Section~4]{CraigSimon83}.
As $E\downarrow\lambda_1$, the integrand is nonnegative and
increases monotonically to
$\log((E_0-\xi)/(\lambda_1-\xi))$.
By monotone convergence and the same formula at $E=\lambda_1$,
we obtain
\[
 \lim_{E\downarrow\lambda_1}L(E)=L(\lambda_1).
\]
Together with $\mu(E)=L(E)$ for $E>\lambda_1$, this yields
$\mu_0=L(\lambda_1)$ in both cases.
Our spectral assumption and the associated speed bound are
\begin{equation}\label{eq:wbar}
 L(\lambda_1)>0,\qquad \wmax=\overline w=\frac{\lambda_1}{L(\lambda_1)}.
\end{equation}
Constants denoted by $C$ may change from line to line; dependence on additional parameters is indicated when needed.

\subsection{Transition fronts}

\begin{definition}\label{def:front}
An entire solution $0<u<1$ of \eqref{eq:kpp} is called a right-moving transition front if there exists $X:\R\to\R$ such that
\begin{equation}\label{eq:uniformends}
 \lim_{R\to+\infty}\sup_{\substack{t\in\R\\x\leq X(t)-R}}|1-u(t,x)|=0,
 \qquad
 \lim_{R\to+\infty}\sup_{\substack{t\in\R\\x\geq X(t)+R}}u(t,x)=0.
\end{equation}
It has positive global mean speed $w$ if
\begin{equation}\label{eq:gms}
 \lim_{T\to+\infty}\sup_{t\in\R}
 \left|\frac{X(t+T)-X(t)}T-w\right|=0,
 \qquad w>0.
\end{equation}
All spatial variables in these formulas belong to $\Sspace$.
\end{definition}
A function $X$ satisfying \eqref{eq:uniformends} is called
an interface location function for $u$.

This definition follows the one-dimensional single-interface
formulation of generalized transition fronts in \citep{BH12}. Any two interface location functions for the same front
differ by a uniformly bounded function, so the global mean
speed is independent of their choice.
In the lattice case, an interface location function may
be chosen integer valued.

\begin{theorem}\label{thm:front}
Assume \eqref{eq:apos} and \eqref{eq:wbar}. For any $g\in\hull$, every transition front of \eqref{eq:kpp} with positive global mean speed $w$ satisfies
\[
 w\leq\wmax.
\]
Moreover, for each $q\in(0,\lambda_1/w)$ there exists $C_q>0$ such that
\begin{equation}\label{eq:fronttail}
 u(t,x)\leq\min\{1,C_qe^{-q(x-X(t))}\},
 \qquad (t,x)\in\R\times\Sspace.
\end{equation}
Consequently, for each fixed $t$,
\[
 \liminf_{x\to+\infty}\frac{-\log u(t,x)}x
 \geq\frac{\lambda_1}{w}\geq L(\lambda_1).
\]
\end{theorem}

\subsection{The Cauchy problem}

Let $u(t,x;u_0,g)$ denote the solution of \eqref{eq:kpp} with initial datum $u_0$. We assume $0\leq u_0\leq1$, with $u_0$ uniformly continuous when $\Sspace=\R$, and
\begin{equation}\label{eq:initialends}
 \lim_{x\to-\infty}u_0(x)=1,
 \qquad \lim_{x\to+\infty}u_0(x)=0.
\end{equation}
{Such initial data are called front-like.}
We usually write $u(t,x)$ for this solution. 
For $\theta\in(0,1)$, define
\begin{equation}\label{eq:levelends}
 x^-_\theta(t)=\inf\{x\in\Sspace:u(t,x)\leq\theta\},
 \qquad x^+_\theta(t)=\sup\{x\in\Sspace:u(t,x)\geq\theta\}.
\end{equation}
{
\begin{definition}\label{def:spreading}
	The lower and upper spreading speeds of the
	$\theta$-level are defined by
	\[
	\underline c_\theta(u_0,g)
	:=\liminf_{t\to+\infty}\frac{x^-_\theta(t)}{t},
	\qquad
	\overline c_\theta(u_0,g)
	:=\limsup_{t\to+\infty}\frac{x^+_\theta(t)}{t}.
	\]
	We say that $u$ has a finite rightward spreading speed
	$c>0$ if
	\begin{equation}\label{eq:spreading}
		\underline c_\theta(u_0,g)
		=\overline c_\theta(u_0,g)=c
		\qquad\text{for every }\theta\in(0,1).
	\end{equation}
\end{definition}
\begin{proposition}[{\citep[Theorem~1.3]{LXZ26}}]
	\label{thm:LXZspreading}
	Let $\Sspace=\R$ and $g\in\hull$.
	Assume that $u_0$ satisfies \eqref{eq:initialends} and
	\[
	u_0(x)\asymp e^{-px}
	\qquad\text{as }x\to+\infty
	\]
	for some $p>0$.
	Then the solution $u(t,x;u_0,g)$ has a finite
	rightward spreading speed $c>0$ independent of $g$. More precisely,
	for every $\varepsilon>0$,
	\[
	\lim_{t\to+\infty}
	\sup_{x\leq(c-\varepsilon)t}|1-u(t,x;u_0,g)|=0,
	\qquad
	\lim_{t\to+\infty}
	\sup_{x\geq(c+\varepsilon)t}u(t,x;u_0,g)=0.
	\]
	Consequently, for every $\theta\in(0,1)$,
	\[
	\lim_{t\to+\infty}\frac{x^-_\theta(t)}{t}
	=
	\lim_{t\to+\infty}\frac{x^+_\theta(t)}{t}
	=c,
	\]
	and hence
	$\underline c_\theta(u_0,g)
	=\overline c_\theta(u_0,g)=c$.
\end{proposition}
}
For $0<s<1/2$, the transition width is
\begin{equation}\label{eq:width}
 W_s(t)=\max\{0,\rs(t)-\ls(t)\}.
\end{equation}
For every fixed finite time, the solution still converges to $1$ as $x\to-\infty$ and to $0$ as $x\to+\infty$, so these quantities are finite. On $\R$, $x^-_\theta$ and $x^+_\theta$ are the extreme points of $E_\theta(t)=\{x:u(t,x)=\theta\}$, as in \citep{LXZ26}. Definition~\eqref{eq:levelends} also applies on $\Z$, where an equality level may be empty. It includes every component of the transition region, without a spatial monotonicity assumption.

\begin{theorem}\label{thm:width}
Assume \eqref{eq:apos} and \eqref{eq:wbar}, and let $g\in\hull$. If $u_0$ satisfies \eqref{eq:initialends} and, for some $s\in(0,1/2)$,
\begin{equation}\label{eq:fastcauchy}
{\overline c_\theta(u_0,g)>\wmax,}
\end{equation}
then
\begin{equation}\label{eq:unboundedwidth}
 \limsup_{t\to+\infty}W_s(t)=+\infty.
\end{equation}
\end{theorem}

\begin{corollary}\label{cor:slow}
Assume \eqref{eq:apos} and \eqref{eq:wbar}. Let $g\in\hull$ and let $u_0$ satisfy \eqref{eq:initialends}. Suppose that, for some $b>0$ and $0<p<L(\lambda_1)$,
\[
 u_0(x)\geq b e^{-p x_+},\qquad x_+=\max\{x,0\}.
\]
Then \eqref{eq:unboundedwidth} holds for every $s\in(0,1/2)$. In particular, this conclusion holds for $u_0(x)=e^{-px_+}$.
\end{corollary}

\begin{remark}\label{rem:sublinearwidth}
{If $u$ has a finite spreading speed in the sense of Definition~\ref{def:spreading},
then, for every $s \in (0,1/2)$ and $\varepsilon > 0$,
\[
0 \leq W_s(t) \leq 2\varepsilon t
\]
for all sufficiently large $t$.
Hence $W_s(t)=o(t)$.} 

{In the continuous case, the solution with
$u_0(x)=e^{-px_+}$, $p>0$, has a finite spreading speed by Proposition~\ref{thm:LXZspreading}.
Consequently, when $0<p<L(\lambda_1)$, Corollary~\ref{cor:slow}
gives
\[
\limsup_{t\to+\infty}W_s(t)=+\infty,
\qquad
\frac{W_s(t)}{t} \to 0
\qquad\text{for every }s\in(0,1/2).
\]
Thus the width is unbounded along a sequence of times,
while remaining sublinear in time.}
\end{remark}

\section{Uniform estimates for the leading edge}\label{sec:local}

We first establish a local exponential growth estimate, uniform with respect to the coefficient in $\hull$ and the position of the interval.

\begin{lemma}\label{lem:boxes}
Let $\lambda_R(g)$ be the principal Dirichlet eigenvalue of $\Lop_g$ on $[-R,R]\cap\Sspace$. In the lattice case, the Dirichlet restriction to
$I_R=[-R,R]\cap\Z$ is defined by extending functions
on $I_R$ by zero to $\Z\setminus I_R$ before applying
$\Lop_g$. Then $\lambda_R(g)\uparrow\lambda_1$ uniformly for $g\in\hull$. The convergence is also uniform with respect to the center of the interval.
\end{lemma}
\begin{proof}
The spectral supremum is the supremum of the Rayleigh quotients of compactly supported test functions, with form domain $H^1(\R)$ in the continuous case. If $\tau_{y_j}g\to\widetilde g$ uniformly, translation invariance of the Rayleigh quotient gives
\[
 \sup\sigma(\Lop_{\widetilde g})\leq\sup\sigma(\Lop_g).
\]
Minimality gives the reverse inequality. Thus $\lambda_1$ is independent of $g$. The variational characterization also shows that $\lambda_R(g)\uparrow\lambda_1$. For fixed $R$, $\lambda_R$ is continuous on the compact hull, since the potentials converge uniformly. Dini's theorem gives uniform convergence, and translation covariance gives the assertion for arbitrary centers. Finally, plateau test functions with increasing support yield $\lambda_1\geq\amin$.
\end{proof}

The next lemma gives a local exponential growth estimate for solutions that remain small on a fixed spatial interval throughout the time interval under consideration.

\begin{lemma}\label{lem:growth}
For every $\varepsilon>0$, there exist $R,C\geq1$ with the following property. If a solution $0<u<1$ of \eqref{eq:kpp} satisfies $u\leq\delta$ on $[t,t+T]\times([y-R,y+R]\cap\Sspace)$, where $T\geq2$, then
\begin{equation}\label{eq:growth}
 u(t+T,y)\geq C^{-1}e^{(\lambda_1-\varepsilon-\amax\delta)T}u(t,y).
\end{equation}
The constants are independent of $g$ and $y$. For Cauchy solutions we take $t\geq1$.
\end{lemma}
\begin{proof}
Choose $R$ by Lemma~\ref{lem:boxes}. In the cylinder under consideration,
$u_t\geq(\Lop_g-\amax\delta)u$.
On $\Z$, the diagonal Dirichlet heat kernel is bounded below by
$e^{\lambda_RT}\varphi_R(y)^2$, where $\varphi_R$ is the positive $\ell^2$-normalized principal eigenfunction. By compactness of the hull, $\varphi_R(y)$ has a uniform positive lower bound.

On $\R$, apply the parabolic Harnack inequality to $u_t=u_{xx}+b(t,x)u$, with $0\leq b\leq\amax$. It gives
\[
 u(t+1,z)\geq C_R^{-1}u(t,y),\qquad |z-y|\leq R.
\]
For the remaining time, compare with the principal Dirichlet eigenfunction normalized by its maximum. Its value at the center is uniformly positive. Absorbing the first unit of time into the constant proves \eqref{eq:growth}. The continuous Harnack and interior estimates used here and below are standard; see \citep{Lieberman}.
\end{proof}

\begin{lemma}\label{lem:displacement}
Let $u$ be a transition front and let $X$ be any associated
interface location function as in Definition~\ref{def:front}. For every $T>0$,
\[
 \sup_{t\in\R}\sup_{0\leq r\leq T}|X(t+r)-X(t)|<+\infty.
\]
If its global mean speed is $w$, then, for each $\varepsilon>0$, there exists $C_\varepsilon$ such that
\begin{equation}\label{eq:displacement}
 |X(t+r)-X(t)-wr|\leq\varepsilon r+C_\varepsilon,
 \qquad t\in\R,\quad r\geq0.
\end{equation}
\end{lemma}
\begin{proof}
We give the proof only for the case \(\mathcal{S}=\mathbb{Z}\); the proof for the case \(\mathcal{S}=\mathbb{R}\) is similar.
    When \(S = \mathbb{Z}\), let \(P_s(n,m)\) be the heat kernel. Then
    \begin{equation}\label{eq:subsuper-solution}
    \sum_m P_r(n,m)u(t,m) \le u(t+r,n) \le e^{\bar{a}r} \sum_m P_r(n,m)u(t,m) 
    \end{equation}
    Assume without loss of generality that
        \begin{equation}\label{eq:uniform-bdd}
        \left\{
        \begin{aligned}
            u(t,n) \ge \frac{2}{3} \qquad &\text{ if } n \le {X(t)-C},  \\
            u(t,n) \le \frac{1}{3} \qquad &\text{ if } n \ge {X(t)+C}, 
        \end{aligned}
        \right. 
    \end{equation} 
    Choose \(\delta\) such that \(e^{\bar{a}T}\delta < 1/3\). Then there exist \(K_\delta, D_\delta\) such that
    \[
    \begin{cases}
        u(t,n) < \delta  \qquad \text{ if } n \ge {X(t)+K_\delta} \\[6pt]
        \sup\limits_{0 \le s \le T} \sum\limits_{|m-n| \ge D_\delta} P_s(n,m) < \delta
    \end{cases}
    \]    
    Thus, when {\(n \ge X(t)+K_\delta+D_\delta\)},
    \[
    \begin{aligned}
        u(t+r,n) &\le e^{\bar{a}r} \sum_m P_r(n,m)u(t,m) \\
        &= e^{\bar{a}r} \sum_{{m < X(t)+K_\delta}} P_r(n,m)u(t,m)
        + e^{\bar{a}r} \sum_{{m \ge X(t)+K_\delta}} P_r(n,m)u(t,m) \\
        &\le e^{\bar{a}r} \sum_{|m-n| \ge D_\delta} P_r(n,m)u(t,m)
        + e^{\bar{a}r} \sum_{{m \ge X(t)+K_\delta}} P_r(n,m){\delta} \\
        &\le 2e^{\bar{a}r}\delta < \frac{2}{3}.
    \end{aligned}
    \]    
    Then from the first inequality in \eqref{eq:uniform-bdd}, we obtain {$X(t)+K_\delta+D_\delta \ge X(t+r)-C$,
    i.e.,
    $$
    X(t+r)-X(t) < K_\delta+D_\delta+C.
    $$}
    Similarly, applying the previous inequality in \eqref{eq:subsuper-solution} and the second inequality in \eqref{eq:uniform-bdd}, one can prove that
    ${X(t+r)-X(t)} \text{ has a lower bound for } t \in \mathbb{R}, \, r \in [0,T].$
    This proves the first assertion. 
    
    For large $r$, \eqref{eq:displacement} follows from the definition of global mean speed; bounded $r$ are covered by the first assertion.
\end{proof}

\begin{proposition}\label{prop:tail}
	Every transition front with positive global mean speed $w$ satisfies \eqref{eq:fronttail}.
\end{proposition}
\begin{proof}
	Fix $q\in(0,\lambda_1/w)$. Choose $\zeta,\varepsilon,\delta>0$ so that
	\[
	b:=\lambda_1-\varepsilon-\amax\delta>0,
	\qquad \frac{b}{w+\zeta}>q.
	\]
	Let $R$ be as in Lemma~\ref{lem:growth}, and choose $K_\delta$ such that $u(t,x)\leq\delta$ for $x\geq X(t)+K_\delta$. By Lemma~\ref{lem:displacement}, the interval $[y-R,y+R]$ remains in this region for elapsed times between zero and
	\[
	T=\frac{y-X(t)-K_\delta-R-C_\zeta}{w+\zeta}.
	\]
	If $T\geq2$, Lemma~\ref{lem:growth} and $u\leq1$ imply $u(t,y)\leq Ce^{-bT}$. Enlarging $C_q$ to cover bounded distances from the interface gives \eqref{eq:fronttail}.
\end{proof}

\section{Properties of the linearized equation}\label{sec:weight}

We next construct a positive decaying weight at almost every element of the hull. The weight solves the spectral-edge equation except at one point; it need not be an eigenfunction.

\begin{lemma}\label{lem:sign}
Let $v$ be a nonzero real solution of $\Lop_gv=\lambda_1v$. If $v$ is square integrable near either spatial infinity, then $v$ has a strict constant sign on $\Sspace$.
\end{lemma}
\begin{proof} Set $A=\lambda_1 I-\Lop_g$. Since $\Lop_g$
is self-adjoint and $\lambda_1=\sup\sigma(\Lop_g)$,
the spectral theorem gives $\sigma(A)\subset[0,\infty)$.
Thus the associated quadratic form satisfies
\[
  Q_A(f)=<f,Af>\geq0,
\]
Here $f\in \ell^2(\mathbb Z)$ in the discrete case
and $f\in H^1(\mathbb R)$ in the continuous case.
Moreover, $Q_A(f)=0$ implies
$f\in\ker A$, and hence $Af=0$.

 First consider that $\Sspace=\Z$ and $v$ is square summable at $+\infty$. For $f=v\mathbf1_{[j+1,+\infty)}$, direct summation gives
\[
 0\leq\ip{f}{(\lambda_1-\Lop_g)f}=v(j)v(j+1).
\]
If equality holds, positivity of the bounded self-adjoint operator $\lambda_1-\Lop_g$ implies $(\lambda_1-\Lop_g)f=0$. Evaluating this equation at $j$ and $j+1$ forces two consecutive values of $v$ to vanish. The recurrence then gives $v\equiv0$, a contradiction. Therefore all consecutive values have the same strict sign.

Then for $\Sspace=\R$, a standard $L^2$ energy estimate gives $v'\in L^2$ on the square-integrable half-line. If $v(y)=0$, extend this half-line portion by zero across $y$. The extension belongs to $H^1(\R)$ and has zero quadratic form for $\lambda_1-\Lop_g$. Hence it is a global weak null solution. Its derivative jump at $y$ vanishes, so $v(y)=v'(y)=0$, which again implies $v\equiv0$. Reflection treats square integrability at $-\infty$ in both cases.
\end{proof}

\begin{proposition}\label{prop:weight}
For $\mu_H$-almost every $g\in\hull$, there exist $\psi>0$ and $d>0$ such that $\psi(0)=1$ and
\begin{equation}\label{eq:defect}
 (\lambda_1-\Lop_g)\psi=d\delta_0,
 \qquad \psi(x)\leq C_\gamma e^{-\gamma|x|}
 \quad\text{for every }0<\gamma<L(\lambda_1).
\end{equation}
On $\Z$, $\delta_0$ is the unit sequence and the adjacent ratios of $\psi$ are bounded.
 On $\R$, the identity is distributional, $\delta_0$ is the Dirac measure, and $\psi$ is continuous, piecewise $C^{2,\alpha}$, and belongs to $H^1(\R)$, with {the logarithmic derivative satisfying $$\left|\frac{\psi'}{\psi}\right|\leq C$$ away from the origin.}
\end{proposition}
\begin{proof}
By Oseledets' multiplicative ergodic theorem
\citep{Oseledets}, applied to the invertible spatial
transfer cocycle at $\lambda_1$, there is a set of
full $\mu_H$-measure on which the cocycle has two
distinct one-dimensional invariant subspaces with
Lyapunov exponents $\pm L(\lambda_1)$.
The corresponding real solutions $v^+$ and $v^-$
decay exponentially at $+\infty$ and $-\infty$,
respectively, at every rate below $L(\lambda_1)$.
In the continuous case we use the continuous-time
version of the theorem for the translation flow.
Lemma~\ref{lem:sign} then allows us to normalize
these solutions by
\[
  v^\pm>0,\qquad v^\pm(0)=1.
\]
Set
\[
 \psi(x)=\begin{cases}v^-(x),&x\leq0,\\v^+(x),&x\geq0.\end{cases}
\]
The residual is supported at zero, with coefficient
\[
 d=\begin{cases}
 \lambda_1+2-g(0)-v^+(1)-v^-(-1),&\Sspace=\Z,\\
 (v^-)'(0)-(v^+)'(0),&\Sspace=\R.
 \end{cases}
\]
Since $\psi(0)=1$, the nonnegative quadratic form yields
$d=\ip{\psi}{(\lambda_1-\Lop_g)\psi}\geq0$, interpreted in the form sense on $\R$. If $d=0$, uniqueness for the second-order equation would identify $v^+$ and $v^-$, contrary to their independence. Thus $d>0$.

The decay follows from the choice of $v^\pm$. On $\R$, the Harnack inequality and the equation give the logarithmic derivative bound. On $\Z$, the equation
\[
 \frac{\psi(n+1)+\psi(n-1)}{\psi(n)}
 =\lambda_1+2-g(n)-d\delta_{n0}
\]
and positivity give bounds for the adjacent ratios in both directions.
\end{proof}

We denote this full-measure set by $\mathcal G\subset\hull$ and call its elements regular. The constants in \eqref{eq:defect} may depend on $g$. The point defect belongs to the test function equation, not to the reaction coefficient.

\begin{lemma}\label{lem:weighted}
Fix $g\in\mathcal G$, and let $X_\psi=L^1(\Sspace,\psi\,dm)$ with norm $\norm{f}_\psi=\mass{\psi|f|}$. Here $m$ denotes Lebesgue measure on $\mathbb R$ or
counting measure on $\mathbb Z$, and $\psi$ is the
weight in \eqref{eq:defect} corresponding to the
fixed phase $g$. The semigroup $P_t=e^{t\Lop_g}$ is positive and strongly continuous on $X_\psi$, and
\begin{equation}\label{eq:semigroup}
 \norm{P_tf}_\psi\leq e^{\lambda_1t}\norm{f}_\psi.
\end{equation}
Let $z$ be a positive mild solution of
\[
  z_t=(\Lop_g-\lambda_1)z
\]
on an open time interval $I$, with
$z\in C(I;X_\psi)$. Then, for any $t_1,t_2\in I$
with $t_1<t_2$,
\begin{equation}\label{eq:massloss}
  \mass{\psi z(t_2)}-\mass{\psi z(t_1)}
  =-d\int_{t_1}^{t_2}z(t,0)\,dt.
\end{equation}
\end{lemma}
\begin{proof}
On $\Z$, the {adjacent ratio bounds of $\psi$} imply that $\Lop_g$ is bounded on $X_\psi$. Absolute convergence justifies index shifts, giving
\[
 \mass{\psi(\Lop_g-\lambda_1)f}=-df(0).
\]
The semigroup is positive by a constant shift of the diagonal. Integration in time proves \eqref{eq:massloss}; \eqref{eq:semigroup} follows first for nonnegative data and then from $|P_tf|\leq P_t|f|$.

On $\R$, the derivative estimate gives $\psi(x)/\psi(y)\leq e^{C|x-y|}$. Gaussian kernel bounds yield a strongly continuous heat semigroup on $X_\psi$, and the bounded potential is a bounded perturbation. Test the equation against $\psi\chi_R$, where $\chi_R=1$ on $[-R,R]$ and vanishes outside $[-2R,2R]$. The cutoff error is bounded by
\[
 C(R^{-1}+R^{-2})\int_{t_1}^{t_2}\norm{z(t)}_\psi\,dt,
\]
which tends to zero {as $R \to +\infty$}. 
This proves \eqref{eq:massloss} for compactly supported data. Positivity, approximation and density give \eqref{eq:semigroup} and the identity for the stated mild solutions. 
At interior times these solutions are locally classical.
\end{proof}

\section{A Liouville theorem for positive ancient solutions}\label{sec:ancient}

In the continuous setting, the following monotonicity property is rooted
in the separation theory for positive ancient solutions
\citep{Pinchover88}.
Lin and Zhang \citep[Remark~2.1]{LinZhang19} explicitly note
the applicability of this approach to time-independent elliptic
operators normalized by their generalized principal eigenvalue.
In our self-adjoint setting, the generalized principal
eigenvalue of $-\Lop_g$ equals
$\inf\sigma_{L^2}(-\Lop_g)=-\lambda_1$,
so their normalization agrees with ours.
We include a self-contained proof covering both the
continuous and lattice settings.
\begin{lemma}\label{lem:ancientmono}
Let $g$ be real and bounded, and locally H\"older continuous when $\Sspace=\R$. Set $\Lop_g=\D+g$ and $\lambda_1=\sup\sigma(\Lop_g)$. Every nonnegative ancient solution of
\begin{equation}\label{eq:ancient}
 z_t=(\Lop_g-\lambda_1)z,
 \qquad (t,x)\in(-\infty,0)\times\Sspace,
\end{equation}
is nondecreasing in time at each point.
\end{lemma}
Here notice that no spatial boundedness or temporal growth assumption is imposed.
\begin{proof}
Let $\mathcal H_+$ be the cone of nonnegative ancient solutions, endowed with local uniform convergence. Fix $z\not\equiv0$. Choose a continuous function $\rho>0$ on $(-\infty,0)$ such that
\[
 0<J_\rho(z):=\int_{-\infty}^0\rho(t)z(t,0)\,dt<+\infty;
\]
for example, take $\rho(t)=e^t/(1+z(t,0))$. After multiplying $z$ by a constant, assume $J_\rho(z)\leq1$ and set
\[
 K_\rho=\{v\in\mathcal H_+:J_\rho(v)\leq1\}.
\]
We first show that $K_\rho$ is compact. Given a compact space-time set $K$, choose a compact time interval $I\subset(-\infty,0)$ lying strictly after $K$. A Harnack chain gives
\[
 v(t,x)\leq C_Kv(r,0),\qquad (t,x)\in K,\quad r\in I.
\]
On $\Z$, the same estimate follows by comparison with the killed heat kernel on a finite interval joining $x$ to zero; the elapsed times are bounded away from zero. Integrating against $\rho$ over $I$ gives a uniform bound on $K$. Interior estimates on $\R$, or the equation and neighboring-site bounds on $\Z$, give local equicontinuity. A diagonal subsequence converges locally to a solution, and Fatou's lemma preserves $J_\rho\leq1$. Thus $K_\rho$ is a compact convex subset of a locally convex space.

Every nonzero element of $\mathcal H_+$ with finite $J_\rho$ has strictly positive $J_\rho$. If $v$ is a nonzero extreme point of $K_\rho$, then $J_\rho(v)=1$; otherwise nearby scalar multiples give a nontrivial convex decomposition. Suppose $v=v_1+v_2$, with $v_i\in\mathcal H_+$. The functionals $J_\rho(v_i)$ are finite, positive for nonzero summands, and additive. Normalization of the summands and extremality show that each $v_i$ is a multiple of $v$. Thus $v$ generates an extreme ray of $\mathcal H_+$.

For every $h>0$, the same-point Harnack estimate gives
\begin{equation}\label{eq:timedomination}
 0\leq v(t-h,x)\leq C_hv(t,x),\qquad t<0,
\end{equation}
uniformly in $t$ and $x$. On $\Z$ this also follows from $v_t\geq-Bv$ for a fixed $B$. The past translate solves the same autonomous equation. The extreme-ray property therefore implies
$v(t-h,x)=k(h)v(t,x)$. Continuity and composition give $k(h)=e^{-\eta h}$ for some $\eta\in\R$, so
\[
 v(t,x)=e^{\eta t}\varphi(x),\qquad
 \varphi>0,\qquad \Lop_g\varphi=(\lambda_1+\eta)\varphi.
\]
Strict positivity follows from the strong maximum principle or the lattice recurrence.

If $\Lop_g\varphi=E\varphi$ with $\varphi>0$, the ground-state identity is
\begin{equation}\label{eq:groundidentity}
 E\norm{f}_2^2-\ip{f}{\Lop_g f}=
 \begin{cases}
 \displaystyle\sum_{n\in\Z}\varphi(n)\varphi(n+1)
 \left(\frac{f(n+1)}{\varphi(n+1)}-\frac{f(n)}{\varphi(n)}\right)^2,&\Sspace=\Z,\\[3mm]
 \displaystyle\int_\R\varphi(x)^2
 \left|\left(\frac{f}{\varphi}\right)'\right|^2dx,&\Sspace=\R,
 \end{cases}
\end{equation}
for finitely supported or smooth compactly supported $f$. Taking the supremum of the Rayleigh quotients shows $E\geq\lambda_1$. Hence $\eta\geq0$, and every extreme point of $K_\rho$ is nondecreasing in time. Time monotonicity is a closed convex condition. The Krein--Milman theorem yields it for all of $K_\rho$, including the chosen solution $z$.
\end{proof}

\begin{corollary}\label{cor:ancient}
For $g\in\mathcal G$, equation \eqref{eq:ancient} has no strictly positive ancient solution with finite $\psi$-weighted mass at an interior time.
\end{corollary}
\begin{proof}
Suppose $M(t_2):=\mass{\psi z(t_2)}<+\infty$ for some $t_2<0$. Lemma~\ref{lem:ancientmono} gives $0\leq z(t,x)\leq z(t_2,x)$ for $t\leq t_2$. In particular, $M(t)$ is finite and nondecreasing on this interval. Test the equation between $t_1<t_2$ against cutoffs of $\psi$. The {adjacent ratio} or logarithmic derivative bounds control the cutoff errors by integrable tails of $\psi z(t_2)$. Passing to the limit gives
\[
 M(t_2)-M(t_1)=-d\int_{t_1}^{t_2}z(t,0)\,dt<0,
\]
a contradiction. 
\end{proof}

\begin{remark}\label{rem:d0}
Fix $g\in\hull$. If a positive exponentially decaying eigenfunction $\phi$ satisfies $\Lop_g\phi=\lambda_1\phi$, the defect is zero and the conclusion of Corollary~\ref{cor:ancient} does not extend to this coefficient: $z(t,x)=\phi(x)$ is a stationary solution. Nevertheless, mass conservation and Lemma~\ref{lem:ancientmono} make every positive finite-$\phi$-mass ancient solution stationary. Its profile is a multiple of $\phi$. Indeed, the Wronskian is constant, and the derivative of the ratio to $\phi$ (or its successive difference) equals this constant divided by $\phi(x)^2$ (or $\phi(n)\phi(n+1)$). Positivity at both infinities and boundedness of $\phi$ force the constant to vanish. A linear solution dominating a front would then have the form $Ce^{\lambda_1t}\phi$, contradicting the front's left limit $1$.

This gives an alternative argument when an edge eigenfunction is assumed. It is distinct from Proposition~\ref{prop:weight}: if an $L^2$ weight satisfied $(\lambda_1-\Lop_g)\psi=d\delta_0$, pairing with $\phi$ would give $d\phi(0)=0$. Such eigenfunctions can therefore occur only outside the regular set used here. We handle those coefficients by translation in Section~\ref{sec:cone}.
\end{remark}

\section{Proof of the upper bound on front speeds}\label{sec:cone}

We first exclude ancient solutions whose right tail remains small ahead of a sufficiently fast line. This will also apply to limits of Cauchy solutions.

\begin{proposition}\label{prop:cone}
Let $g\in\mathcal G$. There is no nonzero bounded ancient solution $0\leq U\leq1$ of \eqref{eq:kpp} for which some $v>\wmax$ satisfies the following property: for every $\delta>0$, there exists $B_\delta$ such that
\begin{equation}\label{eq:cone}
 U(t,x)\leq\delta\qquad\text{if }t\leq0\text{ and }x\geq vt+B_\delta.
\end{equation}
\end{proposition}
\begin{proof}
The strong comparison principle and \eqref{eq:cone} give $0<U<1$ at interior times. Choose
\[
 \frac{\lambda_1}{v}<\gamma<L(\lambda_1).
\]
Next choose $\varepsilon,\delta>0$ so that $b:=\lambda_1-\varepsilon-\amax\delta>\lambda_1/2$, and fix
\begin{equation}\label{eq:betacone}
 \lambda_1<\beta<\min\{2b,\gamma v\},
 \qquad r=\frac{\beta}{b}\in(1,2).
\end{equation}
Let $R$ be as in Lemma~\ref{lem:growth}, and put $C_0=B_\delta+R$. For $x>vt+C_0$, the interval $[x-R,x+R]$ stays in the region where $U\leq\delta$ until time $t+T$, where
\[
 T=\min\{-t,(x-vt-C_0)/v\}.
\]
If $T\geq2$, the local exponential growth estimate of Lemma~\ref{lem:growth} implies $U(t,x)\leq Ce^{-bT}$. Increasing $C$ to include $T<2$ and points behind the line gives
\begin{equation}\label{eq:conetail}
 U(t,x)\leq\min\{1,Ce^{bt}(1+e^{-bx/v})\},\qquad t\leq0.
\end{equation}
For $G(t,x)=g(x)U(t,x)^2$, the inequality $\min\{1,y\}^2\leq y^r$ yields
\begin{equation}\label{eq:remainder}
 \norm{G(t)}_\psi
 \leq Ce^{\beta t}\mass{\psi(x)(1+e^{-bx/v})^r}
 \leq C'e^{\beta t}.
\end{equation}
The spatial integral is finite because $\beta/v<\gamma$. Since $\beta>\lambda_1$,
\[
 \int_{-\infty}^0 e^{-\lambda_1t}\norm{G(t)}_\psi\,dt<+\infty.
\]

For $q<t\leq0$, the variation-of-constants formula gives
\[
 V_q(t):=U(t)+\int_q^tP_{t-r}G(r)\,dr=P_{t-q}U(q).
\]
By \eqref{eq:semigroup} and \eqref{eq:remainder}, $V_q$ converges in $X_\psi$, uniformly on compact time intervals, as $q\to-\infty$. Its limit is
\begin{equation}\label{eq:lift}
 V(t)=U(t)+\int_{-\infty}^tP_{t-r}G(r)\,dr\geq U(t)>0,
\end{equation}
and $V(t)=P_{t-s}V(s)$ for $s<t$. On $\Z$, the generator is bounded on $X_\psi$; on $\R$, local kernel smoothing gives a classical representative. Thus $V_t=\Lop_gV$, and $z(t)=e^{-\lambda_1t}V(t)$ is a positive ancient solution of \eqref{eq:ancient} with finite weighted mass. This contradicts Corollary~\ref{cor:ancient}.
\end{proof}

\begin{lemma}\label{lem:phasetransfer}
If \eqref{eq:kpp} has a transition front with global mean speed $w$ for one $g\in\hull$, then it has a transition front with the same global mean speed for every $\widetilde g\in\hull$.
\end{lemma}
\begin{proof}
Choose $y_j\in\Sspace$ so that $\tau_{y_j}g\to\widetilde g$ uniformly. The mean-speed condition gives $X(t)\to\pm\infty$ as $t\to\pm\infty$. For every $y_j$, the function $u(t,y_j)$ therefore has temporal limits $0$ and $1$. Choose $t_j$ with $u(t_j,y_j)=1/2$. The uniform convergence to the {steady states} in \eqref{eq:uniformends} implies $|X(t_j)-y_j|\leq C$.

Local estimates give a subsequence of $u(t+t_j,x+y_j)$ converging to a solution at $\widetilde g$, with value $1/2$ at $(0,0)$. The strong comparison principle gives values strictly between zero and one. By Lemma~\ref{lem:displacement}, a further diagonal subsequence satisfies
\[
 X(t_j+k)-y_j\longrightarrow\xi(k),\qquad k\in\Z.
\]
Set $\widetilde X(t)=\xi(\lfloor t\rfloor)$. The bounded-time displacement estimate ensures that the limiting solution converges to the {steady states} uniformly in time relative to this interface, as in \eqref{eq:uniformends}. Moreover, for integers $k\leq l$,
\[
 |\xi(l)-\xi(k)-w(l-k)|\leq\varepsilon(l-k)+C_\varepsilon.
\]
Replacing arbitrary times by their integer parts contributes only a bounded error. Dividing by the elapsed time proves the global mean speed $w$.
\end{proof}

\begin{proof}[Proof of Theorem~\ref{thm:front}]
Suppose $w>\wmax$. By Lemma~\ref{lem:phasetransfer}, it suffices to work with $g\in\mathcal G$. Choose $\wmax<v<w$. The lower bound in \eqref{eq:displacement}, applied between $t\leq0$ and zero, gives $X(t)\leq vt+C$. The uniform convergence to zero ahead of the interface in \eqref{eq:uniformends} then implies \eqref{eq:cone}, contradicting Proposition~\ref{prop:cone}. Estimate~\eqref{eq:fronttail} follows from Proposition~\ref{prop:tail}. Taking its logarithm at fixed time and letting $q\uparrow\lambda_1/w$ completes the proof.
\end{proof}

\section{Proof of the unbounded-width result}\label{sec:cauchy}

We argue by contradiction. A bounded transition band gives a continuous interface and uniform right tails. We then select first arrival times of this interface at prescribed spatial positions, retaining a uniform bound on its past positions and obtaining a regular limiting coefficient. The resulting limit satisfies the hypotheses of Proposition~\ref{prop:cone}.

\begin{lemma}\label{lem:band}
Suppose $W_s(t)\leq K$ for $t\geq T_0$ and some $s\in(0,1/2)$. Then:
\begin{enumerate}
\item for every $S>0$, there exists $C_S$ such that
\[
 |\rs(t+r)-\rs(t)|\leq C_S,
 \qquad t\geq T_0,\quad 0\leq r\leq S;
\]
\item for every $\delta>0$, there exists $K_\delta$ such that
\[
 u(t,x)\leq\delta\quad\text{for }t\geq T_0+1,\quad x\geq\rs(t)+K_\delta;
\]
\item there is a continuous function $X(t)$ at bounded distance from $\rs(t)$.
\end{enumerate}
\end{lemma}
\begin{proof}
Increase $K$ by a lattice rounding constant if necessary. Choose $r_0>0$ with $e^{\amax r_0}s<1-s$. Heat comparison, as in Lemma~\ref{lem:displacement}, gives $R>0$ such that, uniformly for $0\leq r\leq r_0$,
\[
 \begin{aligned}
 u(t+r,x)&<1-s&&\text{for }x\geq\rs(t)+R,\\
 u(t+r,x)&>s&&\text{for }x\leq\ls(t)-R.
 \end{aligned}
\]
The first inequality bounds $\ls(t+r)$ from above and hence, by bounded width, bounds $\rs(t+r)$. The second gives a lower bound for $\rs(t+r)$. Iteration proves (1).

For (2), we first claim that, for each $\delta>0$, there is a uniform time $T_\delta$ such that
\begin{equation}\label{eq:pointamplify}
 u(t,x)\geq\delta\quad\Longrightarrow\quad
 u(t+T_\delta,x)\geq s,
 \qquad t\geq T_0+1.
\end{equation}
{We first show that, for every fixed $R>0$, there exists
	$\eta_{\delta,R}>0$, independent of $t$, $x$, and $g$, such that
	\[
	u(t,x)\geq\delta
	\quad\Longrightarrow\quad
	u(t+1,z)\geq\eta_{\delta,R}
	\qquad\text{for }|z-x|\leq R.
	\]
	On $\Z$, this follows by heat comparison from the strict
	positivity of the discrete heat kernel at time one.
	On $\R$, the uniform positive-time gradient bound gives
	$u(t,z)\geq\delta/2$ for $|z-x|\leq\rho_\delta$,
	where $\rho_\delta>0$ is uniform. Heat comparison with this
	bump then yields the same conclusion.}
 Choose a large Dirichlet interval for the homogeneous operator $\D+\amin$, with principal eigenvalue $\lambda_R$ satisfying $\lambda_R/\amin>s$. Normalize its eigenfunction $\varphi_R$ to have maximum one at the center. 
{Set $q_0=\min\{\eta_{\delta,R},s/2\}>0$. Then
	\[
	u(t+1,z)\geq q_0\varphi_R(z-x)
	\qquad\text{on the interval centered at }x.
	\]
	Starting at time $t+1$, compare with the subsolution
	$q(\tau)\varphi_R(z-x)$, where
	\[
	q'=\lambda_Rq-\amin q^2,\qquad q(0)=q_0.
	\]
	Since $q_0<s<\lambda_R/\amin$, $q$ reaches
	$s$ after a uniformly bounded additional time.
	At the center, $\varphi_R(0)=1$, so this proves
	\eqref{eq:pointamplify}.} 
 Combining it with (1) shows that a point with $u(t,x)\geq\delta$ cannot lie to the right of $\rs(t)+C_{T_\delta}+1$. Thus (2) follows.

Finally, let $\chi$ be smooth and nondecreasing, equal to zero on $[0,s]$ and one on $[1-s,1]$. Define
\begin{equation}\label{eq:cutoffinterface}
 X(t)=\mass{\chi(u(t,x))-\mathbf1_{\{x<0\}}}.
\end{equation}
The integrand is compactly supported at each time. Comparison with step functions at $\ls(t)$ and $\rs(t)$ gives $|X(t)-\rs(t)|\leq K+C$. By (1), the supports lie in a common compact interval when $t$ ranges over a compact interval. Dominated convergence gives continuity.
\end{proof}

We first choose a spatial step for which averages along the
discrete sequence of coefficient translates converge to the
hull average.

\begin{lemma}\label{lem:mesh}
There exists $h>0$ such that $\tau_h$ is uniquely ergodic on $\hull$, with invariant measure $\mu_H$.
\end{lemma}
\begin{proof}
If $\Sspace=\Z$, just take $h=1$.

If $\Sspace=\R$, the maps $(\tau_y)_{y\in\R}$ form
a continuous uniquely ergodic flow on the compact
metrizable space $\hull$.
By \citep[Theorem~3.3.34]{FisherHasselblatt19},
$\tau_h$ is uniquely ergodic for all but countably
many $h\in\R$.
Choose such an $h>0$.
Its unique invariant probability measure is $\mu_H$,
because $\mu_H$ is invariant under every spatial
translation.
\end{proof}

\begin{lemma}\label{lem:records}
Assume the conclusions of Lemma~\ref{lem:band}, and let
$\limsup_{t\to+\infty}X(t)/t>v>0$. Then there exist $t_j\to+\infty$ and $y_j\in\Sspace$ such that $X(t_j)=y_j$, $\tau_{y_j}g\to g_*\in\mathcal G$, and, for a fixed $T_1$,
\begin{equation}\label{eq:recordcone}
 X(t)-y_j\leq v(t-t_j)+h,
 \qquad T_1\leq t\leq t_j,
\end{equation}
where $h$ is given by Lemma~\ref{lem:mesh}.
\end{lemma}
\begin{proof}
Choose $w_1$ with $v<w_1<\limsup X(t)/t$, and an integer $n_0$ such that $n_0h>X(T_0+1)$. Let $t_n$ be the first time at which the interface reaches the spatial position $nh$, and define $D_n$ by
\[
 t_n=\inf\{t\geq T_0+1:X(t)\geq nh\},
 \qquad D_n=\frac{nh}{v}-t_n,
 \qquad n\geq n_0.
\]
By continuity, $X(t_n)=nh$; the times $t_n$ are nondecreasing and tend to infinity. The quantity $D_n$ is the advance relative to the reference arrival time $nh/v$. Choose $T_j\to+\infty$ with $X(T_j)\geq w_1T_j$, and put $N_j=\lfloor X(T_j)/h\rfloor$. Since $t_{N_j}\leq T_j$,
\[
 \limsup_{n\to+\infty}\frac{D_n}{n}
 \geq h\left(\frac1v-\frac1{w_1}\right)>0.
\]
Also $D_n-D_{n-1}\leq h/v$. Let $\mathcal R$ consist of $n_0$ and the strict record indices of $D_n$, namely the indices $n>n_0$ for which
\[
 D_n>\max_{n_0\leq m<n}D_m.
\]
Each increase of the running maximum is at most $h/v$, while that maximum grows linearly along a subsequence. Therefore $\mathcal R$ has positive upper density.

If $n\in\mathcal R$ and $n_0\leq m\leq n$, then $t_n-t_m\leq(n-m)h/v$. For $t_{n_0}\leq t\leq t_n$, put $m=\lfloor X(t)/h\rfloor$. Since $t_n$ is the first arrival time at $nh$, we have $m\leq n$. If $m\geq n_0$, the definition of $t_m$ gives $t_m\leq t$, and hence
\[
 X(t)<mh+h\leq nh+v(t-t_n)+h.
\]
If $X(t)<n_0h$, the same estimate follows from $D_n\geq D_{n_0}$ and $t\geq t_{n_0}$. Thus \eqref{eq:recordcone} holds at every record, with $T_1=t_{n_0}$.

It remains to select a regular limiting coefficient without losing
\eqref{eq:recordcone}. Along a subsequence with positive record
density, the finite measures on $\hull$
\[
 \nu_N^{\mathrm{rec}}=\frac1N\sum_{\substack{n_0\leq n\leq N\\n\in\mathcal R}}
 \delta_{\tau_{nh}g}
\]
have a nonzero weak limit $\nu_{\mathrm{rec}}$. Here each point mass is located at a coefficient $\tau_{nh}g\in\hull$; these measures are distinct from the density-of-states measure $\nu$ on $\R$ introduced in Section~\ref{sec:setting}. Testing against nonnegative continuous functions and comparing with the full orbit sum, unique ergodicity gives $\nu_{\mathrm{rec}}\leq\mu_H$. Consequently $\mathcal G$ has full $\nu_{\mathrm{rec}}$-measure. Choose $g_*\in\mathcal G\cap\supp\nu_{\mathrm{rec}}$. Every neighborhood of $g_*$ contains arbitrarily late record coefficients. We can therefore choose record indices $n_j\to+\infty$ with $\tau_{n_jh}g\to g_*$. Taking $y_j=n_jh$ and $t_j=t_{n_j}$ completes the proof. 
\end{proof}

\begin{proof}[Proof of Theorem~\ref{thm:width}]
Suppose \eqref{eq:unboundedwidth} fails. Then $W_s$ is eventually bounded, and Lemma~\ref{lem:band} gives a continuous $X$ at bounded distance from $\rs$. Choose
$\wmax<v<\limsup X(t)/t$ and apply Lemma~\ref{lem:records}. Since $t_j\to+\infty$, the translated solutions are defined on every fixed compact time interval for all sufficiently large $j$. Local compactness gives a subsequence of
\[
 U_j(t,x)=u(t_j+t,y_j+x)
\]
converging to an entire solution $0\leq U\leq1$ at $g_*\in\mathcal G$. A point where $u(t_j,\cdot)\geq s$ stays at bounded distance from $y_j$, so $U\not\equiv0$; on $\R$ use uniform positive-time spatial continuity. The right-tail estimate in Lemma~\ref{lem:band} and \eqref{eq:recordcone} imply
\[
 U(t,x)\leq\delta\qquad\text{for }t\leq0,\quad x\geq vt+B_\delta,
\]
with $B_\delta$ independent of $t$. This contradicts Proposition~\ref{prop:cone}.
\end{proof}

\begin{proof}[Proof of Corollary~\ref{cor:slow}]
We prove a lower spreading estimate directly. Fix $0<v<\lambda_1/p$ and choose $\varepsilon>0$ with $pv<\lambda_1-\varepsilon$. Lemma~\ref{lem:boxes} gives a fixed $R$ such that the principal eigenvalue $\lambda_{y,R}$ on every interval centered at $y$ belongs to $[\lambda_1-\varepsilon,\lambda_1]$. Normalize its positive eigenfunction $\varphi_{y,R}$ by its maximum. Compactness gives $\varphi_{y,R}(y)\geq\eta_R>0$ uniformly in $y$ and $g$.

Decrease $b$ to at most one. On the interval centered at $y$, use the subsolution $q(t)\varphi_{y,R}$, extended by zero, where
\[
 q'=\lambda_{y,R}q-\amax q^2,
 \qquad q(0)=be^{-p(y+R)_+}.
\]
Since $\lambda_{y,R}\leq\amax$, we have $0\leq q\leq1$. Its residual is
$q^2\varphi_{y,R}(g\varphi_{y,R}-\amax)\leq0$; the zero extension has the subsolution sign at the boundary. It is below $u_0$ initially. The logistic formula and $\lambda_{y,R}-pv>0$ show that $q(t)$ is uniformly positive for $y\leq vt$ and large $t$. Thus $u(t,y)\geq\kappa_R>0$ behind $vt$.

For $0<v'<v$, compare over a fixed time with the homogeneous $\amin$ logistic equation on a sufficiently large fixed interval. Starting from $\kappa_R$, its center value can be made larger than any prescribed level below one. For large $t$, every such interval centered at $y\leq v't$ lies behind $v(t-T)$ at the starting time. Hence, by varying $v$ and then $v'$,
\begin{equation}\label{eq:lowerseed}
 \sup_{x\leq v't}|1-u(t,x)|\longrightarrow0,
 \qquad v'<\lambda_1/p.
\end{equation}
The assertion for nonpositive $v'$ follows from any positive intermediate speed. Since $p<L(\lambda_1)$, choose $v'$ between $\wmax$ and $\lambda_1/p$. Then \eqref{eq:fastcauchy} holds for every $s\in(0,1/2)$, and Theorem~\ref{thm:width} applies.
\end{proof}

\begin{remark}\label{rem:upperseed}
For $u_0(x)=e^{-px_+}$ with $0<p<L(\lambda_1)$, comparison with the linearized equation also gives
\[
 u(t,x)\leq e^{[d(p)+\amax]t-px},
 \qquad d(p)=\begin{cases}p^2,&\Sspace=\R,\\2(\cosh p-1),&\Sspace=\Z.\end{cases}
\]
Thus the lower {spreading speed} is at least $\lambda_1/p>\wmax$, and the upper {spreading speed} is finite. An exact common speed is not needed for Corollary~\ref{cor:slow}. 
{
In the continuous case, this existence follows from
Proposition~\ref{thm:LXZspreading}, and
Remark~\ref{rem:sublinearwidth} additionally yields
$W_s(t)=o(t)$.
}
\end{remark}

\begin{remark}[Localized perturbations]\label{rem:localizedwidth}
Consider $u_t=u_{xx}+a(x)f(u)$, where $a\geq1$ is bounded and H\"older continuous, and $a(x)=1$ for $|x|\geq R>0$. Assume $f\in C^2([0,1])$, $f(0)=f(1)=0$, $f'(0)=1$ and $0<f(u)\leq u$ for $0<u<1$. For $0<p<1$, set
\[
 w_p=p+p^{-1},\qquad u_0(x)=\min\{1,e^{-p(x-R)}\}.
\]
Let $\Phi_{w_p}$ be the homogeneous wave. Its tail is asymptotic to $A_pe^{-pz}$, so a sufficiently large $b$ makes $\Phi_{w_p}(x-R+b)\leq u_0(x)$. This wave is a subsolution since $a\geq1$. The function $\min\{1,e^{-p(x-w_pt-R)}\}$ is a supersolution for $t\geq0$: its constant branch covers the perturbation, its exponential branch lies where $a=1$, with residual $v-f(v)\geq0$, and its derivative jump has the supersolution sign. Therefore
\begin{equation}\label{eq:localizedsandwich}
 \Phi_{w_p}(x-w_pt-R+b)\leq u(t,x)
 \leq\min\{1,e^{-p(x-w_pt-R)}\},\qquad t\geq0.
\end{equation}
It follows that $\ls(t)=w_pt+O(1)$, $\rs(t)=w_pt+O(1)$, and $\sup_{t\geq0}W_s(t)<+\infty$ for every fixed $s\in(0,1/2)$.

Under the additional hypotheses of \citep[Theorem~1.4]{NRRZ12}, let $\lambda=\sup\sigma(\partial_{xx}+a)\in(1,2)$. Choosing $p<\sqrt{\lambda-1}$ gives $w_p>\lambda/\sqrt{\lambda-1}$, above the upper bound there for entire fronts. For logistic $f$, a sufficiently strong perturbation with $\lambda>2$ excludes all entire fronts \citep[Theorem~1.2]{NRRZ12}, while the same comparison still applies. Thus nonexistence of fast entire fronts alone does not imply unbounded positive-time width, in contrast with Theorem~\ref{thm:width} for almost periodic media.
\end{remark}

\section{Applications to quasiperiodic media}\label{sec:examples}

We verify $L(\lambda_1)>0$ for analytic quasiperiodic lattice
coefficients at large coupling and for a class of continuous
coefficients with two frequencies. The almost Mathieu model is a
special lattice example for which the spectral Lyapunov exponent is
explicit. In each case, Theorem~\ref{thm:front} and
Corollary~\ref{cor:slow} apply to every phase.

\subsection{Analytic quasiperiodic lattice coefficients at large coupling}
\label{sec:discreteexample}

Let $d\geq1$ and let $v:\T^d\to\R$ be nonconstant and real
analytic. Fix a Diophantine vector $\alpha\in\R^d$: there exist
$\gamma>0$ and $\tau>d$ such that
\[
 \operatorname{dist}(\langle k,\alpha\rangle,\Z)
 \geq\gamma |k|^{-\tau},\qquad k\in\Z^d\setminus\{0\}.
\]
For $\kappa>1$ and $A>\kappa\norm{v}_\infty$, consider
\begin{equation}\label{eq:generaldiscreteKPP}
 u_t(t,n)=\D u(t,n)+a_\theta(n)u(t,n)(1-u(t,n)),
 \qquad a_\theta(n)=A+\kappa v(\theta+n\alpha).
\end{equation}
Here $\Sspace=\Z$, and $\D\varphi(n)=\varphi(n+1)+
\varphi(n-1)-2\varphi(n)$. The coefficients are almost periodic
and satisfy $\inf_n a_\theta(n)>0$.
The rotation $\theta\mapsto\theta+\alpha$ is minimal and uniquely
ergodic. Its image in the coefficient hull has the same properties,
so the framework of Section~\ref{sec:setting} applies.

Let $\Sigma_\kappa^v$ denote the phase-independent spectrum of
$H_{\kappa,\alpha,\theta}^v$ in \eqref{eq:introSchrodinger}, and put
$E_+(\kappa)=\max\Sigma_\kappa^v$. Denote the spatial Lyapunov
exponent of $H_{\kappa,\alpha,\theta}^v\varphi=E\varphi$ by
$L_H(E)$, suppressing the fixed parameters $v$, $\alpha$ and
$\kappa$. Bourgain and Goldstein \citep[Theorem~2]{BG00}
prove that there exists $\kappa_0=\kappa_0(v,\alpha)>1$ such that
\begin{equation}\label{eq:BGpositive}
 L_H(E)>\tfrac12\log\kappa,
 \qquad E\in\R,\quad \kappa>\kappa_0.
\end{equation}
This is the positive-exponent theorem, and no conclusion about
Anderson localization is needed here.
Writing $\Lop_\theta:=\Lop_{a_\theta}$ as in the introduction,
the scalar shift between the operators gives
\begin{equation}\label{eq:discreteshift}
 \Lop_\theta=H_{\kappa,\alpha,\theta}^v+(A-2)I,
 \qquad \lambda_1=A-2+E_+(\kappa),
 \qquad L(\lambda_1)=L_H(E_+(\kappa)).
\end{equation}
In particular, $L(\lambda_1)>\frac12\log\kappa>0$.
Every transition front with positive global mean speed therefore
satisfies
\begin{equation}\label{eq:BGfrontbound}
 w\leq\frac{\lambda_1}{L_H(E_+(\kappa))}
 <\frac{2\lambda_1}{\log\kappa}
 \leq\frac{2(A+\kappa\max_{\T^d}v)}{\log\kappa}.
\end{equation}
The last inequality uses $\sup\sigma(\D+a_\theta)\leq
\sup_n a_\theta(n)$.
For $u_0(n)=e^{-pn_+}$, where $n_+=\max\{n,0\}$ and
$0<p<L_H(E_+(\kappa))$, Corollary~\ref{cor:slow} gives
\begin{equation}\label{eq:BGwidth}
 \limsup_{t\to+\infty}W_s(t)=+\infty,
 \qquad 0<s<1/2.
\end{equation}
In particular, $0<p<\frac12\log\kappa$ is sufficient.
These statements hold for every $\theta\in\T^d$, with the fixed
frequency vector above and $\kappa>\kappa_0$.

Take $d=1$ and $v(\theta)=2\cos(2\pi\theta)$. For this special
potential the Diophantine restriction is unnecessary: let
$\alpha\in\R\setminus\Q$, $\kappa>1$ and $A>2\kappa$, and set
\begin{equation}\label{eq:amo}
 a_\theta(n)=A+2\kappa\cos\bigl(2\pi(\theta+n\alpha)\bigr).
\end{equation}
Write $H_{\kappa,\alpha,\theta}=H_{\kappa,\alpha,\theta}^v$.
Its Lyapunov exponent equals $\log\kappa$ throughout its spectrum
\citep{Avila15,BJ02}. Thus \eqref{eq:discreteshift} gives
$L(\lambda_1)=\log\kappa$, and
\begin{equation}\label{eq:amobounds}
 w\leq\wmax=\frac{\lambda_1}{\log\kappa}
\end{equation}
for every transition front with positive global mean speed.
For $u_0(n)=e^{-pn_+}$ with $0<p<\log\kappa$,
\begin{equation}\label{eq:amopropagation}
 \limsup_{t\to+\infty}W_s(t)=+\infty,
 \qquad 0<s<1/2.
\end{equation}

These conclusions are independent of the spectral type.
For every irrational $\alpha$, evenness of the sampling function
and \citep{JS94} give purely singular continuous spectrum on a
dense $G_\delta$ set of phases when $\kappa>1$.
There are also regimes in which the spectrum is purely singular
continuous for every phase. Let $q_j$ be the continued-fraction
denominators of $\alpha$, and define
$ \beta(\alpha)=\limsup_{j\to\infty}
                 \frac{\log q_{j+1}}{q_j}.$
If $0<\beta(\alpha)\leq \infty$ and
$1<\kappa<e^{\beta(\alpha)}$,  the spectrum
is again purely singular continuous \citep{AYZ17}.
Thus neither \eqref{eq:amobounds} nor \eqref{eq:amopropagation}
requires a localized state at the spectral edge, or indeed any
$\ell^2$ eigenfunction. The condition that $\alpha$ is Liouvillean
alone does not imply the quantitative inequalities used here.

\subsection{Continuous quasiperiodic coefficients with two frequencies}
\label{sec:continuousexample}

Let $\alpha\in\R\setminus\Q$ satisfy
\[
 \operatorname{dist}(m\alpha,\Z)\geq\gamma |m|^{-\tau},
 \qquad m\in\Z\setminus\{0\},
\]
for some $\gamma>0$ and $\tau>1$. Let $V:\T^2\to\R$ be real
analytic, nonnegative, and attain its minimum value $0$ at only
finitely many points. For $\kappa>0$, $A>\kappa\max_{\T^2}V$
and $\theta=(\theta_1,\theta_2)\in\T^2$, set
\begin{equation}\label{eq:continuousgeneral}
 \begin{aligned}
 a_{\kappa,\theta}(x)&=A-\kappa V(\theta_1+x,\theta_2+\alpha x),\\
 H_{\kappa,\theta}&=-\partial_{xx}
             +\kappa V(\theta_1+x,\theta_2+\alpha x).
 \end{aligned}
\end{equation}
The reaction coefficient is uniformly positive. The irrational
flow with direction $(1,\alpha)$ is minimal and uniquely ergodic,
and the spectrum $\Sigma_\kappa$ of $H_{\kappa,\theta}$ is
independent of $\theta$. Write
$E_0(\kappa)=\inf\Sigma_\kappa$ and let $L_H(E)$ denote the
Lyapunov exponent of $H_{\kappa,\theta}\varphi=E\varphi$ per
unit length in $x$, with $V$, $\alpha$ and $\kappa$ fixed.

We use the low-energy part of You and Zhou's phase-transition
theorem \citep[Theorem~1.1]{YZ14}. More precisely, the
positive-exponent estimate used in its proof in Section~5,
based on Bjerkl\"ov's work \citep{Bjerklov06}, states that for
any sufficiently small fixed $\varepsilon>0$ there exist
$c_0>0$ and $\kappa_0>0$ such that
\begin{equation}\label{eq:YZpositive}
 L_H(E)\geq c_0\sqrt\kappa,
 \qquad E\in\Sigma_\kappa\cap(-\infty,\kappa^{1-\varepsilon}],
 \quad \kappa>\kappa_0.
\end{equation}
The constants can depend on $V$, $\alpha$ and $\varepsilon$.
It is this pointwise estimate that verifies the spectral hypothesis
of our main results.

For completeness, the spectral bottom lies in this low-energy
interval for all sufficiently large $\kappa$. Choose a minimum
$\theta_*\in\T^2$. Since $V(\theta_*)=0$ and
$\nabla V(\theta_*)=0$, Taylor's theorem gives
\[
 0\leq V(\theta_*+(x,\alpha x))\leq Cx^2
\]
for $|x|$ sufficiently small. Fix a real function
$\eta\in C_c^\infty((-1,1))$ with $\norm\eta_{L^2}=1$, and put
$\eta_\kappa(x)=\kappa^{1/8}\eta(\kappa^{1/4}x)$.
Then $\norm{\eta_\kappa}_{L^2}=1$, and, at phase $\theta_*$,
\[
 \begin{aligned}
 \langle\eta_\kappa,H_{\kappa,\theta_*}\eta_\kappa\rangle
 &\leq \int_\R |\eta_\kappa'(x)|^2\,dx
       +C\kappa\int_\R x^2|\eta_\kappa(x)|^2\,dx\\
 &\leq C_1\sqrt\kappa.
 \end{aligned}
\]
Nonnegativity, the variational principle and phase independence
of the spectrum imply
\begin{equation}\label{eq:continuumbottom}
 0\leq E_0(\kappa)\leq C_1\sqrt\kappa.
\end{equation}
Choose $\varepsilon<1/2$ as above. For large $\kappa$,
$C_1\sqrt\kappa<\kappa^{1-\varepsilon}$, so
\eqref{eq:YZpositive} applies at $E_0(\kappa)$.

Write $\Lop_{\kappa,\theta}:=\Lop_{a_{\kappa,\theta}}$.
The sign convention is important: our linearized KPP operator satisfies
\begin{equation}\label{eq:reflection}
 \Lop_{\kappa,\theta}=AI-H_{\kappa,\theta},\qquad
 \lambda_1=A-E_0(\kappa),\qquad
 L(\lambda_1)=L_H(E_0(\kappa))\geq c_0\sqrt\kappa>0.
\end{equation}
The equality of Lyapunov exponents follows because
$\Lop_{\kappa,\theta}\varphi=\lambda\varphi$ and
$H_{\kappa,\theta}\varphi=(A-\lambda)\varphi$ are the same
second-order equation, with the same spatial variable.
Consequently, every transition front with positive global mean
speed satisfies
\begin{equation}\label{eq:YZfrontbound}
 w\leq\frac{A-E_0(\kappa)}{L_H(E_0(\kappa))}
 \leq\frac{A}{c_0\sqrt\kappa}.
\end{equation}
For $u_0(x)=e^{-px_+}$ with $0<p<L_H(E_0(\kappa))$,
{Corollary~\ref{cor:slow} and Remark~\ref{rem:sublinearwidth} gives
\[
 \limsup_{t\to+\infty}W_s(t)=+\infty,
 \qquad
 \frac{W_s(t)}{t} \to 0,
 \qquad 0<s<1/2.
\]}
In particular, $0<p<c_0\sqrt\kappa$ is sufficient. All these
conclusions hold for every phase, once $\kappa$ is sufficiently
large.

\section{Further remarks}\label{sec:scope}

The strict inequality $w>\lambda_1/L(\lambda_1)$ is used in \eqref{eq:betacone}. Thus existence of a front at the upper endpoint remains open under our assumptions. For spreading solutions, it also remains to determine whether $W_s(t)$ tends to infinity and whether its growth rate can be estimated.

The proof uses self-adjointness in the sign argument, the defect pairing and the ground-state identity. It also uses the one-dimensional second-order structure to construct the two half-line solutions. For nonlocal dispersal or advection, corresponding adjoint and Harnack arguments would be needed. A constant diffusion coefficient $D>0$ can be included by replacing $\D$ with $D\D$ throughout, with $\lambda_1$ and $L(E)$ defined for that operator.

Almost periodicity enters in two different ways: minimality and compactness give a uniform local exponential growth estimate, while unique ergodicity allows the spatial translates associated with selected first arrival times to converge to a regular coefficient. The record indices in Lemma~\ref{lem:records} ensure that this selection retains the bound on past interface positions. The argument is deterministic and requires neither a pure point spectral decomposition nor localization of the parabolic semigroup. In particular, the singular continuous example in Section~\ref{sec:examples} shows that a localized eigenfunction is not the mechanism assumed by the theorem. The distinction from localized perturbations is illustrated by Remark~\ref{rem:localizedwidth}.

 \section*{Declaration on the Use of AI}

The human authors initiated and led the research program, formulated the mathematical problem, developed the main results, and independently verified all arguments. Through discussions with ChatGPT, the human authors obtained numerous references and explanations of their content; some of these references were relevant to the present work and proved useful in developing the proofs. ChatGPT also assisted in improving notation, grammar, and presentation. The final manuscript was carefully revised and approved by the human authors, who take full responsibility for its content.

\end{document}